\documentclass[a4paper,11pt]{amsart}
\usepackage{latexsym}
\usepackage[english]{babel}
\usepackage{amssymb, amsmath, amsthm}

\newtheorem{thm}{Theorem}
\newtheorem{prop}{Proposition}
\newtheorem{cor}{Corollary}
\newtheorem{lem}{Lemma}

\theoremstyle{remark}
\newtheorem{rem}{Remark}

\newcommand{\N}{\mathbb{N}}

\newenvironment{iterate}{\bigskip  \noindent \hspace*{\stretch{1}} \begin{minipage}[b]{0.9\textwidth} }{\end{minipage} \hspace*{\stretch{1}} \bigskip  }

\title[Non-typical points for $\beta$-shifts]{Non-typical points for $\boldsymbol \beta$-shifts}

\date{\today}

\author{David F\"arm}
\address{David F\"arm, Institute of Mathematics, Polish Academy of
  Sciences ulica \'Sniadeckich 8, P.O. Box 21, 00-956 Warszawa,
  Poland}
\curraddr{Centre for Mathematical Sciences, Lund
  University, Box 118, 22100 Lund, Sweden}
\email{davidfarm@gmail.com}

\author{Tomas Persson}
\address{Tomas Persson, Institute of Mathematics, Polish Academy of
  Sciences, ulica \'Sniadeckich 8, P.O. Box 21, 00-956 Warszawa,
  Poland}
\curraddr{Centre for Mathematical Sciences, Lund
  University, Box 118, 22100 Lund, Sweden}
\email{tomasp@maths.lth.se}

\thanks{Both authors were supported by EC FP6 Marie Curie ToK
  programme CODY. Part of the paper was written when the authors were
  visiting institut Mittag-Leffler in Djursholm in 2010. The authors
  are grateful for the hospitality of the institute.}


\begin{document}

\begin{abstract}
  We study sets of nontypical points under the map $f_\beta \mapsto
  \beta x $ mod 1, for non-integer $\beta$ and extend our results from
  \cite{farm} in several directions. In particular we prove that sets
  of points whose forward orbit avoid certain Cantor sets, and set of
  points for which ergodic averages diverge, have large intersection
  properties. We observe that the technical condition $\beta>1.541$
  found in \cite{farm} can be removed.
\end{abstract}

\maketitle

\section{$\beta$-shifts}
Let $[x]$ denote the integer part of the real number $x$, and let
$\lfloor x \rfloor$ denote the largest integer strictly smaller than
$x$. Let $\beta > 1$. For any $x \in [0, 1]$ we associate the sequence
$d(x,\beta) = (d (x, \beta)_n )_{n=0}^\infty \in \{0, 1, \ldots,
\lfloor \beta \rfloor \}^\N$ defined by
\[
  d(x, \beta)_n :=
  [\beta f_\beta^{n} (x)],
\]
where $f_\beta (x) = \beta x \pmod 1$.
The closure, with respect to the product topology, of the set
\[
  \{\, d (x,\beta) : x \in [0,1)\,\}
\]
is denoted by $S_\beta$ and it is called the $\beta$-shift. We will
denote the set of all finite words occuring in $S_\beta$ by
$S_\beta^*$.  The sets $S_\beta$ and $S_\beta^*$ are invariant under
the left-shift $\sigma \colon (i_n)_{n=0}^\infty \mapsto (i_{n+1}
)_{n=0}^\infty$ and the map $d (\cdot, \beta) \colon x \mapsto d (x,
\beta)$ satisfies the equality $\sigma^n ( d (x, \beta) ) = d (
f_\beta^n (x), \beta)$. If we order $S_\beta$ with the lexicographical
ordering then the map $d ( \cdot, \beta)$ is one-to-one and monotone
increasing. Let $d_{-}(1,\beta)$ be the limit in the product topology
of $d(x,\beta)$ as $x$ approaches 1 from below. Then the subshift
$S_\beta$ satisfies
\begin{equation} \label{eq:Sbeta}
  S_\beta = \{\, (j_k)_{k=0}^\infty : \sigma^n (j_k)_{k=0}^\infty \leq
  d_{-} (1, \beta) \ \forall n \,\}.
\end{equation}
Note that $d_{-} (1, \beta)=d(1, \beta)$ if and only if $d(1,\beta)$
contains infinitely many non-zero digits.

Parry proved in \cite{Parry1} that the map $\beta \mapsto d (1,
\beta)$ is monotone increasing and injective. For a sequence
$(j_k)_{k=0}^\infty$ there is a $\beta > 1$ such that
$(j_k)_{k=0}^\infty = d (1, \beta)$ if and only if $\sigma^n
((j_k)_{k=0}^\infty) < (j_k)_{k=0}^\infty$ for every $n > 0$. The
number $\beta$ is then the unique positive solution of the equation
\[
  1 = \sum_{k=0}^\infty \frac{d_k (1,\beta)}{ x^{k + 1}}.
\]
One observes that the fact that the map $\beta \mapsto d (1,
\beta)$ is monotone increasing and injective together with
(\ref{eq:Sbeta}) imply that $S_{\beta_1} \subseteq S_{\beta_2}$ holds
if and only if $\beta_1 \leq \beta_2$.

If $x \in [0,1]$ then
\[
  x = \sum_{k=0}^\infty \frac{d_k (x, \beta) }{\beta^{k + 1}}.
\]
This formula can be seen as an expansion of $x$ in the non-integer
base $\beta$, and thereby generalises the ordinary expansion in
integer bases.

We let $\pi_\beta$ be the map $\pi_\beta \colon S_\beta \to [0,1)$
  defined by
\[
  \pi_\beta \colon (i_k)_{k=0}^\infty \quad \mapsto \quad
  \sum_{k=0}^\infty \frac{i_k}{\beta^{k + 1}}.
\]
Hence, $\pi_\beta ( d(x, \beta)) = x$ holds for any $x \in [0,1)$ and
  $\beta > 1$.

We define cylinder sets as  
\[
  [i_0 \cdots i_{n-1}] := \{\, (j_k)_{k=0}^\infty \in S_\beta : i_k =
  j_k,\ 0\leq k < n \,\},
\]
and say that $n$ is the generation of the cylinder $[i_0 \cdots
  i_{n-1}]$. We will also call the half-open interval $\pi_\beta ([i_0
  \cdots i_{n-1}])$ a cylinder of generation $n$. The set $[i_0 \cdots
  i_{n-2}]$ will be called the parent cylinder of $[i_0 \cdots
  i_{n-1}]$.

Note that if $d(1,\beta)$ has only finitely many non-zero digits, then
$S_\beta$ is a subshift of finite type, so there is a constant $C>0$
such that
\begin{equation}\label{SFT}
  C \beta^{-n}\leq |\pi_\beta ([i_0 \cdots i_{n-1}])|\leq \beta^{-n}.
\end{equation}

\section{Transversality and large intersection classes}

In \cite{falconer}, Falconer defined $\mathcal{G}^s$, $0<s\leq n$, to
be the class of $G_\delta$ sets $F$ in $\mathbb{R}^n$ such that
$\dim_H(\cap_{i=1}^\infty f_i(F))\geq s$ for all sequences of
similarity transformations $(f_i)_{i=1}^\infty$. He characterised
$\mathcal{G}^s$ in several equivalent ways and proved among other
things that countable intersections of sets in $\mathcal{G}^s$ are
also in $\mathcal{G}^s$.

In \cite{farm}, the following approximation theorem was proven, where
$\mathcal{G}^s$ are restrictions of Falconer's classes to the unit
interval.
 
\begin{thm} \label{oldthm}
  Let $\beta \in (1.541,2)$ and let $(\beta_n)_{n=1}^\infty$ be any
  sequence with $\beta_n\in (1.541,\beta)$ for all $n$, such that
  $\beta_n \to \beta$ as $n\to \infty$. Assume that $E\subset S_\beta$
  and $\pi_{\beta_n}(E \cap S_{\beta_n})$ is in the class
  $\mathcal{G}^s$ for all $n$. If $F$ is a $G_\delta$ set such that $F
  \supset \pi_\beta(E)$, then $F$ is also in the class
  $\mathcal{G}^s$.
\end{thm}

When expanding a number $x$ in base $\beta>1$ as
$d(x,\beta)=(x_k)_{k=0}^\infty$, one can consider how often a given
word $y_1 \dots y_m$ occurs. If the expression
\[
  \frac{ \#\{i \in \{0, \dots, n-1\} : x_i \dots  x_{i+m-1}=y_1
    \dots  y_m \} } {n}
\]
converges as $n\to \infty$, it gives an asymptotic frequency of the
occurence of the word $y_1 \dots y_m$ in the expansion of $x$ to the
base $\beta$. Theorem \ref{oldthm} was used in \cite{farm} to prove
the following.

\begin{prop}\label{oldprop}
  For any sequence of bases $(\beta_n)_{n=1}^\infty$, such that
  $\beta_n \in (1.541,2)$ for all $n$, the set of points for which the
  frequency of any finite word does not converge in the expansion to
  any of these bases, has Hausdorff dimension 1.
\end{prop}

The reason for the condition $\beta \in (1.541,2)$ in
Theorem~\ref{oldthm} and Proposition~\ref{oldprop} is that in
\cite{farm}, we needed some estimates on the map
\begin{equation} \label{holder}
  \sum_{k=1}^\infty \frac{a_k-b_k}{\beta_1^k} \mapsto
  \sum_{k=1}^\infty \frac{a_k-b_k}{\beta_2^k}, \quad (a_1,a_2
  \dots),(b_1,b_2 \dots) \in S_{\beta_1},
\end{equation}
when $\beta_1<\beta_2$, provided by the following transversality lemma
by Solomyak \cite{solomyak}.

\begin{lem}\label{solomyak}
Let $x_0<0.649$. There exists a constant $\delta>0$ such that if $x\in
[0,x_0]$ then
\[
  |g(x)| < \delta \quad \Longrightarrow \quad g'(x) < -\delta
\]
holds for any function of the form
\begin{equation}
  g (x) = 1+\sum_{k = 1}^\infty a_k x^k, \ \textrm{where} \ a_k\in
  \{-1,0,1\}.\label{functions}
\end{equation}
\end{lem}

The condition $x_0 < 0.649$ in Lemma~\ref{solomyak} introduces the
condition $\beta > 1 / 0.649$ or for simplicity $\beta>1.541$. But,
when studying the map defined in (\ref{holder}), the coefficients in
the power series (\ref{functions}) will not be free to take values in
$\{-1,0,1\}$, they will be the difference of two sequences from
$S_\beta$. This allows us to remove the condition $x_0 < 0.649$, which
is done by using Lemma~\ref{improvedsolomyak} below instead of
Lemma~\ref{solomyak}.

\begin{lem}\label{improvedsolomyak}
  Let $\beta > 1$. There exists a constant $\delta>0$ such that if
  $x\in [0,1/\beta]$ then
  \[
    |g(x)| < \delta \quad \Longrightarrow \quad g'(x) < -\delta
  \]
  holds for any function of the form
  \[
    g (x) = 1+\sum_{k = 1}^\infty (a_k-b_k) x^k, \quad \textrm{where}
    \ (a_1,a_2 \dots),(b_1,b_2 \dots) \in S_\beta.
  \]
\end{lem}

This lemma was stated and proved in \cite{farmpersson}, were it was
used for other purposes. We refere to \cite{farmpersson} for the
proof, were in fact, the lemma was proved with the condition $x\in
[0,1/\beta]$ replaced by the weaker condition $x\in [0,1/\beta +
  \varepsilon]$, where $\varepsilon$ is a small positive constant. In
this note we will however only need the weaker form stated above.

Replacing Lemma~\ref{solomyak} by Lemma~\ref{improvedsolomyak} in the
proofs of \cite{farm}, we immediately get the following improved
versions of Theorem~\ref{oldthm} and Proposition~\ref{oldprop}. Note
that allowing $\beta>1$ instead of $\beta\in (1,2)$ only affects
notation slightly by adding new symbols to the shift space
$S_\beta$. The proofs in \cite{farm} go through almost verbatim. Also
the result from \cite{farm2}, which is used in \cite{farm} to prove
Proposition~\ref{oldprop}, is easily extended from $\beta\in (1,2)$ to
$\beta>1$.

\begin{thm}
  Let $\beta >1$ and let $(\beta_n)_{n=1}^\infty$ be any sequence with
  $\beta_n<\beta$ for all $n$, such that $\beta_n \to \beta$ as $n\to
  \infty$. Assume that $E\subset S_\beta$ and $\pi_{\beta_n}(E \cap
  S_{\beta_n})$ is in the class $\mathcal{G}^s$ for all $n$. If $F$ is
  a $G_\delta$ set such that $F \supset \pi_\beta(E)$, then $F$ is
  also in the class $\mathcal{G}^s$.
\end{thm}

\begin{prop}
  For any sequence of bases $(\beta_n)_{n=1}^\infty$, such that
  $\beta_n>1$ for all $n$, the set of points for which the frequency
  of any finite word does not converge in the expansion to any of
  these bases, has Hausdorff dimension 1.
\end{prop}

\section{Schmidt games and avoiding Cantor sets}

In \cite{schmidt}, Schmidt introduced a set-theoretic game which can
be seen as a metric version of the Banach--Mazur game (see for example
\cite{oxtoby}). We present here a modified version of Schmidt's game
that was used in \cite{farm}.

Consider the unit interval $[0,1]$ with the usual metric and a set
$E\subset [0,1]$. Two players, Black and White, play the game in
$[0,1]$ with two parameters $0<\alpha,\gamma<1$ according to the
following rules:

\begin{iterate}
  {\em In the initial step} Black chooses a closed interval
  $B_0\subset [0,1]$.
  \smallskip

  {\em Then the following step is repeated.} At step $k$, White
  chooses a closed interval $W_k\subset B_k$ such that $|W_k|\geq
  \alpha |B_k|$. Then Black chooses a closed interval $B_{k+1}\subset
  W_k$ such that $|B_{k+1}|\geq \gamma |W_k|$.
\end{iterate}

We say that $E$ is $(\alpha,\gamma)$-winning if there is a strategy
that White can use to make sure that $\bigcap_k W_k \subset E$, and
$\alpha$-winning if this holds for all $\gamma$. As was shown in
\cite{farm}, the following proposition easily follows from the methods
in \cite{schmidt}.

\begin{prop} \label{winningproperties}~
  \begin{itemize}
    \item[a.] If $E$ is $\alpha$-winning for $\alpha=\alpha_0$, then
      $E$ is $\alpha$-winning for all $\alpha\leq \alpha_0$.

    \item[b.] If $E_i$ is $\alpha$-winning for $i=1,2,3,\dots$, then
      $\cap_{i=1}^\infty E_i$ is also $\alpha$-winning.

    \item[c.] If $E$ is $\alpha$-winning, then the Hausdorff dimension
      of $E$ is 1.
  \end{itemize}
\end{prop}

In \cite{farm}, the following proposition was proven.
\begin{prop}\label{oldwinning}
  For any $\beta \in (1,2)$ and any $x\in [0,1]$,
  \[
    G_{\beta}(x) = \Big\{\, y \in [0,1] : x \notin
    \overline{\bigcup_{n=0}^\infty f_{\beta}^n(y)} \, \Big\}.
  \] 
  is $\alpha$-winning for any $\alpha \leq 1/16$.
\end{prop}

The set $G_{\beta} (x)$ consists of points for which the forward orbit
under $f_\beta$ is bounded away from $x$. One can also think of
$G_{\beta} (x)$ as the union over all $\delta > 0$, of sets of points
with orbits not falling into a hole of radius $\delta$ around $x$.

Let us at this point compare our result with a result by
Dolgopyat\footnote{We are grateful to a referee for pointing out this
  paper to us.}  \cite{dolgopyat}. Dolgopyat proved that if $E$ is a
set of Hausdorff dimension strictly smaller than 1, and $f$ is a
piecewise expanding map on an interval, then the set of points, for
which the orbit under $f$ avoids the set $E$, has full Hausdorff
dimension. Dolgopyat's result is stronger in the sence that the result
holds for a much larger class of maps. However, Dolgopyat's result
does not give any intersection property, and in this sence our result
is stronger, since we can treat countably many different maps at the
same time, whereas Dolgopyat's result only gives results for one fixed
map.

Here, we will extend our results in the spirit of Dolgopyat, and
instead of considering only orbits avoiding a point, we consider
orbits avoiding a more general set $E$. In doing so, we will prove
that the set of points that avoid a set $E$ is $\alpha$-winning, and
so get a stronger statement than only full Hausdorff dimension, which
is the result of Dolgopyat.

However, we will need to impose some restrictions on the set $E$. More
precisely, we will prove the following proposition which shows that we
can avoid entire Cantor sets instead of just single points. We
consider sets of the form
\[
G_{f_{\beta}}(\pi_\beta(\Sigma_A))=\Big\{\, y\in [0,1):
  \pi_{\beta}(\Sigma_A) \, \, \cap \overline{\bigcup_{n = 0}^\infty
    f_{\beta}^n(y)}=\emptyset \, \Big\},
\]
where $\Sigma_A$ denotes a subshift of finite type. We then have that
$\pi_\beta (\Sigma_A)$ is a Cantor set in $[0,1]$. Hence the set
$G_{f_{\beta}} (\pi_\beta(\Sigma_A))$ is the set of points with
forward orbit bounded away from the Cantor set $\pi_\beta
(\Sigma_A)$. One can also think of the set $G_{f_{\beta}}
(\pi_\beta(\Sigma_A))$ as the union over all $\delta > 0$, of sets of
points with orbits not falling into a hole consisting of a
$\delta$-neighbourhood of $\pi_\beta (\Sigma_A)$.

\begin{prop}\label{propwinning}
  Let $\beta >1$ and let $\Sigma_A \subset S_\beta$ be a subshift of
  finite type, such that there is a finite word $i_0 \dots i_n$ from 
  $S_\beta \setminus \Sigma_A$. Then there exist $\alpha>0$ such
  that
  \[
    G_{f_{\beta}}(\pi_\beta(\Sigma_A))=\Big\{\, y\in [0,1):
      \pi_{\beta}(\Sigma_A) \, \, \cap \overline{\bigcup_{n =
          0}^\infty f_{\beta}^n(y)}=\emptyset \, \Big\}
  \]
  is $\alpha$-winning.
\end{prop}

A quick look at Proposition \ref{winningproperties} gives us the
following corollary.

\begin{cor}\label{fulldim}
  Let $N\in \N$, $\beta_1, \dots, \beta_N >1$ and for each
  $1\leq n \leq N$, let $\Sigma_{A_n} \subset S_{\beta_n}$ such that
  there is a finite word $i_1 \dots i_{k_n}$ from  $S_{\beta_n}
  \setminus \Sigma_{A_n}$. The the set
  \[
    \bigcap_{n=1}^N G_{f_{\beta_n}}(\pi_\beta(\Sigma_{A_n}))
  \]
  has Hausdorff dimension 1.
\end{cor}

The reason why $N$ in Corollary \ref{fulldim} must be finite is that
$\alpha_0$ from Proposition~\ref{propwinning} will depend on $\beta$
and $\Sigma_A$. When taking intersections we need a uniform $\alpha$
for which the sets are $\alpha$-winning, to be able to say anything
about the intersection. See Remark \ref{explicitalpha} at the end of
the paper for an estimate of $\alpha$. If we have uniform estimates on
the $\alpha$, then we can take countable intersections in
Corollary~\ref{fulldim}.

Before giving the proof of Proposition \ref{propwinning}, we note that
if $S_\beta$ is a subshift of finite type, then
Proposition~\ref{propwinning} is easy to prove.
Indeed, then there is a constant $C > 0$ such that
\begin{equation} \label{eq:estforSFT}
C \leq \frac{| \pi_\beta ([i_0 \ldots i_k])|}{\beta^{k+1}} \leq 1
\end{equation}
holds for all cylinders $[i_0 \ldots i_k]$. Using
\eqref{eq:estforSFT}, it is not hard to see that there is an $\alpha_0
> 0$ such that each time White plays she can introduce the word $i_0
\ldots i_n$, that is missing in $\Sigma_A$. By \eqref{eq:estforSFT}
this implies that the word $i_0 \ldots i_n$ occurs regularly in $\{y\}
= \cap_k W_k$, and this means that $\bigcup_{n=0}^\infty f_\beta^n
(y)$ is bounded away from $\pi_\beta (\Sigma_A)$. Hence
Proposition~\ref{propwinning} need only be proved in the case when
$S_\beta$ is not of finite type.

The case when $S_\beta$ is not of finite type is much more difficult,
since we have no uniform lower bound on the size of cylinders, such as
(\ref{SFT}). The key step in proving Proposition \ref{oldwinning} was
the following theorem from \cite{farm}. It will be used in the proof
of Proposition~\ref{propwinning}.

\begin{thm}\label{thmapprox}
  Let $\beta \in (1,2)$ and let $(\beta_n)_{n = 1}^\infty$ be any
  sequence with $\beta_n \in (1, \beta)$ for all $n$ such that
  $\beta_n \to \beta$ as $n \to \infty$. Let also $E \subset S_\beta$
  and $\alpha \in (0, 1)$. If $\pi_{\beta_n} (E \cap S_{\beta_n})$ is
  $\alpha$-winning for $\alpha = \alpha_0$ for all $n$, then
  $\pi_\beta (E)$ is $\alpha$-winning for any $\alpha \leq \min \{
  \frac{1}{16}, \frac{\alpha_0}{4} \}$.
\end{thm}

\begin{rem} \label{allbeta}
The condition $\beta \in (1,2)$ in Theorem \ref{thmapprox} comes from
the fact that in \cite{farm}, we chose to work with $\beta<2$ to
simplify notation. It is not difficult to extend the proof of Theorem
\ref{thmapprox} to hold for all $\beta>1$.

The only place in which $\beta \in (1,2)$ was used is in what is
called ``An auxiliary strategy''. There we use the fact that in any
cylinder $\pi_\beta([i_0 \dots i_n])$, the player White needs at most
a factor $2$ to make sure that the game continues in $\pi_\beta([i_0
  \dots i_n0])$, thereby avoiding the cylinder $\pi_\beta([i_0 \dots
  i_n1])$ which may have bad properties. If $\beta>2$, a factor $2$ is
still enough for White to avoid the cylinder $\pi_\beta([i_0 \dots
  i_n\lfloor \beta \rfloor])$ which may have bad properties. The
factor $2$ is not enough for White to choose any other cylinder
$\pi_\beta([i_0 \dots i_n k])$ in one move, but after a couple of
moves, the game is already played in such a small set that at most two
of these cylinders remain, so White can pick at least one of
them. That is all what is needed for the strategy to work.
\end{rem}

Proposition \ref{propwinning} follows from Theorem \ref{thmapprox} and
Remark \ref{allbeta} once we have proven the following proposition.

\begin{prop}\label{approx}
  Let $\beta> 1$ such that $S_\beta$ is not of finite type and
  let $\Sigma_A\subset S_\beta$ be a subshift of finite type. Then
  there exist $\alpha>0$ and $\beta_0<\beta$ such that
  \[
    G_{\beta'}(\pi_{\beta'}(\Sigma_A))=\Big\{\, y\in [0,1):
      \pi_{\beta'}(\Sigma_A) \, \, \cap \overline{\bigcup_{n=0}^\infty
        f_{\beta'}^n(y)}=\emptyset \, \Big\}
  \]
  is $\alpha$-winning for any $\beta'\in[\beta_0,\beta]$ such that
  $S_{\beta'}$ is of finite type.
\end{prop}

To prove Proposition \ref{approx}, we need some lemmata.

\begin{lem}\label{scaling}
  Let $\beta> 1$ and let $i_0 \dots i_n$ be a finite word in
  $S_\beta^*$ such that we have $i_0 \dots i_n j_0 \dots j_m \in
  S_\beta^*$ for all finite words $j_0 \dots j_m \in S_\beta^*$. Then
  we have $| \pi([i_0 \dots i_n])| = \beta^{-n-1}$ and 
  \[
  | \pi_\beta([i_0 \dots i_n j_0 \dots
    j_m])|=\beta^{-n-1} |\pi_\beta([j_0 \dots j_m])|
  \]
  for all finite words $j_0 \dots j_m\in S_\beta^*$.
\end{lem}
 
\begin{proof}
  It is clear that $\sigma^{n+1} ([i_0 \dots i_n])=S_\beta$, so
  $f_\beta^{n+1} (\pi_\beta([i_1 \dots i_n]))=[0,1)$, where
    $f_\beta^{n+1}$ is just a scaling with factor $\beta^{n+1}$ on
    $\pi_\beta([i_0 \dots i_n])$. Thus, $\pi_\beta([i_0 \dots i_n])$
    is just a smaller copy of $[0,1)$.
\end{proof}

\begin{lem}\label{lowerbound}
  Let $\beta>1$, $M \in \mathbb N$ and $k\in \mathbb N$ be such that
  $(d(1, \beta)_n)_{n = 0}^{M} 0^k 1 \in S_{\beta}^*$. If $\beta_0 \in
  (1, \beta)$ is such that $(d(1, \beta)_n)_{n = 0}^{M} 0^k 1 \in
  S_{\beta_0}^*$, then for all $i_0 \dots i_n \in S_{\beta_0}^*$ such
  that $M = \max \{\, m : i_{n - m} \dots i_n = (d(1, \beta)_n)_{n =
    0}^{m} \,\}$, it holds that
  \[
    |\pi_{\beta'}([i_0 \dots i_n])|\geq \beta^{-(n+k+2)}, \, \mbox{for
      all } \, \beta'\in [\beta_0,\beta].
  \]
\end{lem}

\begin{proof}
  Let $\beta'\in [\beta_0,\beta]$. From (\ref{eq:Sbeta}) and the
  maximality of $M$ we conclude that $i_0 \dots i_{n-M} j_0 \dots j_m
  \in S_{\beta'}^*$ for all $j_0 \dots j_m \in S_{\beta'}^*$. From
  Lemma~\ref{scaling} we then get $|\pi_\beta([i_0 \dots i_n])|\geq
  |\pi_\beta([i_0 \dots i_n0^{k+1}])| = \beta^{-(n+k+2)}$.
\end{proof}

\begin{lem}\label{ineachinterval}
  Let $\beta> 1$ and $M\in \mathbb N$. There exist $\epsilon>0$
  and $\beta_0<\beta$ such that for any $\beta'\in[\beta_0,\beta]$ and
  for any interval $I\subset [0,1]$, there exists a cylinder
  $\pi_{\beta'}([i_0 \dots i_n])$ such that $\max \{m : i_{n-m} \dots
  i_n=(d(1,\beta)_n)_{n=0}^{m}\}\geq M$ for which
  $|\pi_{\beta'}([i_0 \dots i_n])\cap I|>\epsilon |I|$. Moreover, if
  $S_{\beta'}$ is of finite type, then $|\pi_{\beta'}([i_0 \dots
    i_n])\cap I|>\sigma_{\beta'} |\pi_{\beta'}([i_0 \dots i_n])|$,
  where $\sigma_{\beta'}>0$ is independent of $I$.
\end{lem}

\begin{proof}
  Let $\beta' \in [\beta_0,\beta]$ as in Lemma \ref{lowerbound} and
  let $I\subset [0,1]$ be an interval. Note that all cylinders in this
  proof will be with respect to $S_{\beta'}$. Let $n$ be the smallest
  generation for which there is a cylinder contained in $I$. Let
  $\pi_{\beta'}([i_0 \dots i_{n-1}])$ be one of these generation $n$
  cylinders in $I$. By the minimality of $n$ we know that the parent
  cylinder, $\pi_{\beta'}([i_0 \dots i_{n-2}])$ covers at least one
  endpoint of $I$. If $\pi_{\beta'}([i_0 \dots i_{n-2}])$ does not
  cover $I$, let $m$ be the smallest generation for which there is a
  cylinder contained in $I \setminus \pi_{\beta'}([i_0 \dots
    i_{n-2}])$. Let $\pi_{\beta'}([j_0 \dots j_{m-1}])$ be one of
  these generation $m$ cylinders. By the minimality of $m$ we know
  that the parent cylinder, $\pi_{\beta'}([j_0 \dots j_{m-2}])$ covers
  the other endpoint of $I$.

  Together, the cylinders $\pi_{\beta'}([i_0 \dots i_{n-2}])$ and
  $\pi_{\beta'}([j_0 \dots j_{m-2}])$ cover $I$. Indeed, if not, then
  there is a smallest generation $l$ for which there is a cylinder
  $\pi_{\beta'}([k_0 \dots k_{l-1}])$ between $\pi_{\beta'}([i_0 \dots
    i_{n-2}])$ and $\pi_{\beta'}([j_0 \dots j_{m-2}])$. Consider its
  parent cylinder $\pi_{\beta'}([k_0 \dots k_{l-2}])$. If
  $\pi_{\beta'}([k_0 \dots k_{l-2}])$ would intersect one of
  $\pi_{\beta'}([i_0 \dots i_{n-2}])$ and $\pi_{\beta'}([j_0 \dots
    j_{m-2}])$, then it would have to contain it. But this is
  impossible since the minimality of $n$ and $m$ implies $l\geq
  n,m$. Thus, $\pi_{\beta'}([k_0 \dots k_{l-2}])$ is also between
  $\pi_{\beta'}([i_0 \dots i_{n-2}])$ and $\pi_{\beta'}([j_0 \dots
    j_{m-2}])$, which contradicts the minimality of $l$.

  Consider the one of $\pi_{\beta'}([i_0 \dots i_{n-2}])$ and
  $\pi_{\beta'}([j_0 \dots j_{m-2}])$ that covers at least half of
  $I$. Let us assume it is $\pi_{\beta'}([i_0 \dots i_{n-2}])$ but it
  makes no difference for the argument.

  If $\max \{\, m : i_{n-m-2} \dots i_{n-2} = (d(1, \beta)_n)_{n =
    0}^{m} \,\} \geq M$, then we can choose the set $\pi_{\beta'}([i_0
    \dots i_{n-2}])\cap I$ as long as $\epsilon\leq 1/2$ and we are
  done with the first claim. The second claim, that
  $|\pi_{\beta'}([i_0 \dots i_{n-2}])\cap I|>\sigma_{\beta'}
  |\pi_{\beta'}([i_0 \dots i_{n-2}])|$ follows from the fact that
  $|\pi_{\beta'}([i_0 \dots i_{n-1}]) \subset I$ and (\ref{SFT}),
  since $S_{\beta'}$ is of finite type.

  Assume instead that $\max \{m : i_{n-m-2} \dots
  i_{n-2}=(d(1,\beta)_n)_{n=0}^{m}\}=N< M$. Then $i_0 \dots i_{n-2}
  (d(1,\beta)_n)_{M-N-1}^{M} \in S_{\beta'}^*$ by
  (\ref{eq:Sbeta}). By Lemma \ref{lowerbound} there is a $k$ that only
  depends on $\beta$ and $M$ such that
  \begin{align*}
    |\pi_{\beta'}([i_0 \dots i_{n-2} (d(1,\beta)_n)_{M-N-1}^{M}])| &\geq
    \beta^{-(n+M+k+1)} \\ &\geq \beta^{-(M+k+2)}|\pi_{\beta'}([i_0
      \dots i_{n-2}])|.
  \end{align*}

  Since $| \pi_{\beta'} ([i_0 \ldots i_{n-2}])| \geq |I| / 2$, we
  conclude that if $\epsilon \leq \beta^{-(M+k+2)}/2$, then we can
  choose the cylinder $\pi_{\beta'}([i_0 \dots i_{n-2}
    (d(1,\beta)_n)_{M-N-1}^{M}])\subset I$. This ensures the truth of
  both claims and we are done.
\end{proof}

We are now ready to prove Proposition \ref{approx}.

\begin{proof}[Proof of Proposition \ref{approx}]
  Note that since $\Sigma_A$ is of finite type while $S_\beta$ is not,
  there is an $M>1$ such that $(d(1,\beta)_n)_{n=0}^{M}$ is not
  allowed in $\Sigma_A$. For this $M$ choose $\epsilon$ and $\beta_0$
  as in Lemma \ref{ineachinterval}. Let $\beta'\in [\beta_0,\beta]$
  such that $S_{\beta'}$ is of finite type, and let
  $\alpha=\epsilon/2$.

  Assume that Black has chosen his first interval $B_0$. We will
  construct a strategy that White can use to make sure that $\cap_k
  W_k =\{x\}\subset G_{\beta'}(\Sigma_A)$, or equivalently that
  $(f^n_\beta(x))_{n=0}^\infty$ is bounded away from
  $\pi_{\beta'}(\Sigma_A)$.

  Each time Black has chosen an interval $B_k$,
  Lemma~\ref{ineachinterval} ensures that White can choose $W_k
  \subset \pi_{\beta'} [i_0 \dots i_n] \cap B_k$, where
  \[
  \max \{\, m :
  i_{n-m} \dots i_n = (d(1, \beta)_n)_{n = 0}^{m} \,\} \geq M
  \]
  and $|W_k| \geq \sigma(\beta') | \pi_{\beta'} [i_0 \dots
    i_n]|$. Since $\beta' < \beta$, there is an $N$ such that
  $(d(1,\beta)_n)_{n=0}^{N}\notin S_{\beta'}$. It implies that for the
  cylinders $\pi_{\beta'} ([i_0 \dots i_n])$ that occur here, the
  numbers $\max \{\, m : i_{n-m} \dots i_n = (d_n (1,\beta))_{n=0}^m
  \,\}$ will be bounded by $N$.

  If White plays like this, it ensures that the sequence $d(x,\beta')$
  contains the word $(d(1,\beta)_n)_{n=0}^{M}$ regularly. Thus,
  $f^n_\beta(x)$ is always in a cylinder outside
  $\pi_{\beta'}(\Sigma_A)$. If $f^n_\beta(x)$ would be bounded away
  from the endpoints of these cylinders, then
  $(f^n_\beta(x))_{n=0}^\infty$ would be bounded away from
  $\pi_{\beta'}(\Sigma_A)$. But $\alpha=\epsilon/2$, so there is a
  factor 2 left after placing $W_k$ in $\pi_{\beta'}[i_0 \dots i_n]
  \cap B_k$. White can place $W_k$ in the middle of $\pi_{\beta'}[i_0
    \dots i_n] \cap B_k$, thereby avoiding the endpoints.

  We conclude that $G_{\beta'}(\Sigma_A)$ is $\alpha$-winning and we
  are done.
\end{proof}

\begin{rem}\label{explicitalpha}
  The $\alpha$ in Proposition \ref{approx} can be extracted quite
  easily from the proofs.  Let $M$ be such that $(d_k (1, \beta))_{k =
    0}^M$ is not at word in $\Sigma_A$. Take $k$ such that $( d_j (1,
  \beta) )_{j = 0}^M 0^k 1 < d (1, \beta)$. Then $\alpha= \beta^{- (M
    + k + 1)} / 4$ is small enough.  It follows that in Proposition
  \ref{propwinning}, $\alpha= \beta^{- (M + k + 1)} /16$ is small
  enough. Note that these values for $\alpha$ are not optimal, but
  they make it possible to extend Corollary \ref{fulldim} to countable
  intersections, for some cases.
\end{rem}

\end{document}